\documentclass[12pt]{amsart}
\usepackage{mathrsfs}
\usepackage{amsmath}
\usepackage{amsfonts}
\usepackage{amssymb}
\usepackage[all]{xy}           
\usepackage{bm}
\usepackage{bbding}
\usepackage{txfonts}
\usepackage{amscd}
\usepackage{xspace}
\usepackage[shortlabels]{enumitem}
\usepackage{ifpdf}

\ifpdf
  \usepackage[colorlinks,final,backref=page,hyperindex]{hyperref}
\else
  \usepackage[colorlinks,final,backref=page,hyperindex,hypertex]{hyperref}
\fi
\usepackage{tikz}
\usepackage[active]{srcltx}

\makeatletter

\newtheorem{thm}{Theorem}[section]
\newtheorem{lem}[thm]{Lemma}

\newtheorem{pro}[thm]{Proposition}
\newtheorem{ex}[thm]{Example}
\theoremstyle{definition}
\newtheorem{rmk}[thm]{Remark}
\newtheorem{defi}[thm]{Definition}

\newcommand{\nc}{\newcommand}
\newcommand{\delete}[1]{}

\nc{\mlabel}[1]{\label{#1}}  
\nc{\mcite}[1]{\cite{#1}}  
\nc{\mref}[1]{\ref{#1}}  
\nc{\mbibitem}[1]{\bibitem{#1}} 

\delete{
\nc{\mlabel}[1]{\label{#1}{\hfill \hspace{1cm}{\bf{{\ }\hfill(#1)}}}}
\nc{\mcite}[1]{\cite{#1}{{\em{{\ }(#1)}}}}  
\nc{\mref}[1]{\ref{#1}{{\em{{\ }(#1)}}}}  
\nc{\mbibitem}[1]{\bibitem[\em #1]{#1}} 
}

\newcommand {\emptycomment}[1]{}

\nc{\oprn}{\theta}
\nc{\Oprn}{\Theta}

\nc{\calo}{\mathcal{O}}
\nc{\oop}{$\mathcal{O}$-operator\xspace}
\nc{\oops}{$\mathcal{O}$-operators\xspace}
\nc{\mrho}{{\bm{\varrho}}}
\nc{\emk}{\mathbf{K}}
\nc{\invlim}{\displaystyle{\lim_{\longleftarrow}}\,}
\nc{\ot}{\otimes}

\newcommand{\be }{\begin{equation}}
\newcommand{\ee }{\end{equation}}

\newcommand{\g}{\mathfrak g}
\newcommand{\h}{\mathfrak h}

\newcommand{\huaG}{\mathcal{G}}

\newcommand{\huaX}{\mathcal{X}}
\newcommand{\huaY}{\mathcal{Y}}

\newcommand{\huaD}{\mathcal{D}}

\newcommand{\huaO}{{\mathcal{O}}}

\newcommand{\frkD}{\mathfrak D}

\newcommand{\frkL}{\mathfrak L}

\newcommand{\frkR}{\mathfrak R}

\newcommand{\half}{\frac{1}{2}}

\newcommand{\Courant}[1]{\left\llbracket  #1\right\rrbracket }

\newcommand{\Id}{{\rm{Id}}}

\newcommand{\br}[1]{   [ \cdot,    \cdot  ]   }

\newcommand{\Hom}{\mathrm{Hom}}

\newcommand{\gl}{\mathfrak {gl}}

\newcommand{\de}{\mathrm{d}}

\nc{\CV}{\mathbf{C}}

\begin{document}

\title{Cohomology for Deformation maps on quasi-twilled Lie triple systems}

\author{Jia Zhao}
\address{Jia Zhao, School of Mathematics and Statistics, Nantong University, Nantng 226019, Jiangsu, China}
\email{zhaojia@ntu.edu.cn}

\date{\today}

\begin{abstract}
Due to the importance of invariant, in the present paper, we provide a unified approach to study various kinds of operators on Lie triple systems such as crossed homomorphisms, relative Rota-Baxter operators, Reynolds operators etc. For this purpose, we examine two kinds of deformation maps on quasi-twilled Lie triple systems. A quasi-twilled Lie triple system is a Lie triple system which can be decomposed into two subspaces, one of which is a subalgebra. Moreover, we establish cohomologies of two kinds of deformation maps, which recover those of several kinds of operators on Lie triple systems.
\end{abstract}


\keywords{quasi-twilled Lie triple system, deformation map, cohomology}

\maketitle

\vspace{-1.1cm}

\tableofcontents

\allowdisplaybreaks

 \section{Introduction}
 Cartan's study on symmetric spaces \cite{Cartan} can be seen as the original study on Lie triple systems. Later Jacobson abstracted an algebraic object and named it as Lie triple systems \cite{Jacobson}, and then Hodge and Parshall defined its representation \cite{Hodge}. Lie triple systems play a crucial role in Lie theory, which is reflected in that they are related with other algebra structures closely. For instance, a Lie triple system is a special Nambu algebra, while the space of fundamental objects of Lie triple systems is endowed with a Leibniz algebra structure. Due to the importance of Lie triple systems, a structure theory of Lie triple systems was established by Lister \cite{Lister}. Moreover, applications of Lie triple systems on numerical analysis of differential equations were studied in \cite{Kaas} and that $T^*$-extension of Lie triple systems is compatible with nilpotency and solvability was examined in \cite{Deng}.

 As we all know, seeking for a invariant is a classical method to study a mathematical object. Among these invariants, cohomology has many applications in topology, geometry, and algebra. In application of algebra aspects, cohomology controls deformations and extensions. As for deformations, we refer that Gerstenhaber explored formal deformations on associative algebras \cite{Ge1,Ge2}, while Nijenhuis and Richardson examined deformations on Lie algebras \cite{NR1,NR2}. Study on deformations of other algebra structures and quadratic operads was deeply studied by Balavoine \cite{Bal}. Besides, study on cohomology and deformations of maps on special algebra is a complement to that of algebra structures. Recently, a series works on operators such as morphisms \cite{Bar,Bor,Fre1,Fre2,Fre3}, derivations \cite{TFS}, $\huaO$-operators (also called relative Rota-Baxter operators) \cite{HST,T.S2,TBGS,Uh1}, crossed homomorphisms \cite{JS,PSTZ}, twisted Rota-Baxter operators, and Reynolds operators \cite{Das1,Das2} on Lie algebras, Leibniz algebras, associative algebras, or even 3-Lie algebras were examined. Moreover, Jiang, Sheng, and Tang provided a unified approach to study cohomology of various operators on Lie algebras \cite{JST}.

 Since we introduced the notion of quasi-twilled Lie triple systems and explored relationship between relative Rota-Baxter operators and mathced pairs of Lie triple systems using the twisting of Lie triple systems in \cite{ZXQ}, motivated by Jiang, Sheng and Tang's work on Lie algebras, we provide a method to study cohomology of various operators on Lie triple systems. For this purpose, we first recall the notion of quasi-twilled Lie triple systems. A quasi-twilled Lie triple system is a Lie triple system which can be decomposed into two subspaces, one of which is a subalgebra. Different from the case of Lie algebras, as a Lie triple system structure is a ternary operation, the operation $\Theta$ of a twilled Lie triple system $(\huaG,\g_1,\g_2)$ should be written as $\Theta=\hat\phi_1+\hat\mu_1+\hat\psi+\hat\mu_2$, and it is more convenient to consider the case of strict ones, i.e. $\hat\psi=0$. Then we are ready to introduce two types of deformation maps. As we can see in the present paper, deformation map of type I on some special quasi-twilled Lie triple systems recovers crossed homomorphisms of any weight, derivations, and homomorphisms on Lie triple systems, while deformation map of type II on some quasi-twilled Lie triple systems recovers relative Rota-Baxter operators of any weight, relative Rota-Baxter operators, twisted Rota-Baxter operators, Reynolds operatoes on Lie triple systems and deformation map of matched pair of Lie triple systems. Besides, we also establish cohomologies of two types of deformation maps. Consequently, we obtain cohomologies of several operators on Lie triple systems. As was stated before, cohomology has a close connection with deformations in that the infinitesimal deformations of maps can be classified by the first cohomology groups. In order to build cohomologies of these deformation maps and use it to study deformations of several operators on Lie triple systems, we should thus construct the 0-cocycles of two kinds of deformation maps. We explored cohomologies and deformations of crossed homomorphisms and relative Rota-Baxter operators on Lie-Yamaguti algebras \cite{ZQ2,ZQX}, and one obtains those on Lie triple systems if the binary operation on Lie-Yamaguti algebras is restricted to zero. These two types of deformation  maps recovers some maps we have not studied before, and this work provides a unified method to study maps on higher structures.

 The paper is organized as follows. In Section 2, we first recall some basic concepts such as Lie triple systems, their representations, and cohomology of Lie triple systems. Then we recall the notion of quasi-twilled Lie triple systems and we give some special examples. In Section 3, we introduce the notion of deformation map of type I, and illustrate that crossed homomorphisms of any weight, derivations and homomorphisms are special cases of deformation map of type I on some special quasi-twilled Lie triple systems as was given in Section 2. Moreover, we build the cohomology of deformation map of type I so that we obtain that of crossed homomorphisms of any weight, derivations and homomorphisms on Lie triple systems. Parallel to Section 3, we introduce the notion of deformation maps of type II in Section 4. We also point out that relative Rota-Baxter operators (of any weight), twisted Rota-Baxter operators and Reynolds operators on Lie triple systems, and deformation map of matched pair of Lie triple systems are special cases of deformation maps of type II. Cohomology of deformation maps of type II is also established, and consequently we obtain that of relative Rota-Baxter operators (of any weight), twisted Rota-Baxter operators and Reynolds operators on Lie triple systems, and deformation map of matched pair of Lie triple systems.

 In this paper, all the vector spaces are over $\mathbb{K}$, a field of characteristic $0$.

\section{Preliminaries}
This section reviews some basic concepts and conclusions.

\subsection{Backgrounds}
This subsection reviews some concepts such as Lie triple systems, their representations, and cohomology.
\begin{defi}\cite{Jacobson}\label{LY}
A {\bf Lie triple system} is a vector space $\g$, together with a trilinear bracket $\Courant{\cdot,\cdot,\cdot}_\g:\wedge^2\g \otimes  \mathfrak{g} \to \mathfrak{g} $ such that the following equations are satisfied for all $x,y,z,w,t \in \g$,
\begin{eqnarray}
~ &&\label{LY1}\Courant{x,y,z}_\g+\Courant{y,z,x}_\g+\Courant{z,x,y}_\g=0,\\
~ &&\Courant{x,y,\Courant{z,w,t}_\g}_\g=\Courant{\Courant{x,y,z}_\g,w,t}_\g+\Courant{z,\Courant{x,y,w}_\g,t}_\g+\Courant{z,w,\Courant{x,y,t}_\g}_\g.\label{fundamental}
\end{eqnarray}
We denote a Lie triple system by a pair $(\g,\Courant{\cdot,\cdot,\cdot}_\g)$.
\end{defi}

Note that sometimes we write Eq. \eqref{LY1} as $\Courant{x,y,z}_\g+c.p.=0$ in the present paper.

\begin{defi}\cite{Yamaguti3}
Let $(\g,\Courant{\cdot,\cdot,\cdot}_\g)$ be a Lie triple system and $V$ a vector space. A {\bf representation} of $\g$ on $V$ consists of a bilinear map $\rho:\otimes^2 \g \to \gl(V)$ such that for all $x,y,z,w \in \g$,
\begin{eqnarray}
~&&\label{RYT4}\rho(z,w)\rho(x,y)-\rho(y,w)\rho(x,z)-\rho(x,\Courant{y,z,w}_\g)+D_\rho(y,z)\rho(x,w)=0,\\
~&&\label{RLY5}\rho(\Courant{x,y,z}_\g,w)+\rho(z,\Courant{x,y,w}_\g)=[D_\rho(x,y),\rho(z,w)],
\end{eqnarray}
where the bilinear map $D_\rho:\otimes^2\g \to \gl(V)$ is given by
\begin{eqnarray}\label{rep}
D_\rho(x,y):=\rho(y,x)-\rho(x,y), \quad \forall x,y \in \g.
\end{eqnarray}
It is obvious that $D_\rho$ is skew-symmetric. We denote a representation of $\g$ on $V$ by $(V;\rho)$. In order to distinguish from the deformation map of type I in the present paper, the subscript ``$\rho$" in $D_\rho$ will not be omitted in this paper.
\end{defi}
\emptycomment{
\begin{rmk}
By \eqref{RLY5} and \eqref{rep}, one can by a direct computation deduce that
\begin{eqnarray}
[D(x,y),D(z,w)]=D(\Courant{x,y,z},w)+D(z,\Courant{x,y,w}), \quad \forall x,y,z,w \in \g.\label{rep3}
\end{eqnarray}
\end{rmk}
}

\begin{ex}\label{ad}
Let $(\g,\Courant{\cdot,\cdot,\cdot}_\g)$ be a Lie triple system. For any $x,y\in \g$, we define $\frkR :\otimes^2\g \to \gl(\g)$ by $(x,y) \mapsto \frkR_{x,y}$ , where $\frkR_{x,y}z=\Courant{z,x,y}_\g$ for all $z \in \g$. Then $(\g;\frkR)$ forms a representation of $\g$ on itself, called the {\bf adjoint representation}. Moreover,  for all $x,y \in \g$, by Eq. \eqref{LY1}, $\frkL_{x,y}\triangleq D_\frkR(x,y)=\frkR_{y,x}-\frkR_{x,y}$ is given by,
\begin{eqnarray*}
\frkL_{x,y}z=\Courant{x,y,z}_\g, \quad \forall z \in \g.\label{lef}
\end{eqnarray*}
\end{ex}

In the sequel, let us recall the cohomology of Lie triple systems. Let $(\g,\Courant{\cdot,\cdot,\cdot}_\g)$ be a Lie triple system and $(V;\rho)$ a representation of $\g$ on $V$. The set of $n$-cochains is defined to be
\begin{eqnarray*}
C^{n}_{\mathsf{Ymg}}(\g;V):=\Hom(\otimes^{2n-1} \g,V)=\Hom\Big(\underbrace{(\ot^2\g)\ot\cdots\ot(\ot^2\g)}_{n-1}\ot\g,V\Big), \quad n \geq 1.
\end{eqnarray*}

If $n\geq 2,$ $f \in C^{n}_{\mathsf{Ymg}}(\g;V)$ satisfies
\begin{eqnarray*}
f(\huaX_1,\cdots ,\huaX_{n-2},x,x,y)=0,\\
f(\huaX_1,\cdots ,\huaX_{n-2},x,y,z)+f(\huaX_1,\cdots,\huaX_{n-2} ,y,z,x)+f(\huaX_1,\cdots,\huaX_{n-2}, z,x,y)=0,
\end{eqnarray*}
for all $x,y,z\in \g$ and $\huaX_i=x_i\ot y_i\in\ot^2\g$, where $1 \leq i \leq n-2$.

For any $f \in C^n_{\mathsf{Ymg}}(\g; V),~ (n\geq 1)$, the coboundary map $\de :C^n_{\mathsf{Ymg}}(\g;V)\to C^{n+1}_{\mathsf{Ymg}}(\g;V)$ is defined by
\begin{eqnarray*}
~ &&(\de f)(\huaX_1,\cdots,\huaX_{n},x)\\
~ &=&(-1)^{n+1}\Big(\rho(y_{n},x)f(\huaX_1,\cdots,\huaX_{n-1},x_{n})-\rho(x_{n},x)f(\huaX_1,\cdots,\huaX_{n-1},y_{n})\Big)\\
~ &&+\sum_{i=1}^{n} (-1)^{i+1}D_\rho(\huaX_i)f(\huaX_1,\cdots,\hat{\huaX_{i}},\cdots,\huaX_{n},x)\\
~ &&+\sum_{i=1}^{n} (-1)^{i}f(\huaX_1,\cdots,\hat{\huaX_{i}},\cdots,\huaX_{n+1},\Courant{x_i,y_i,x}_\g)\\
~ &&+\sum_{j<k} (-1)^j f(\huaX_1,\cdots,\hat{\huaX_{j}},\cdots,\huaX_{k-1},[\huaX_j,\huaX_k]_F,\huaX_{k+1},\cdots,\huaX_{n},x),
\end{eqnarray*}
for all $\huaX_i=x_i\ot y_i\in\ot^2\g,~i=1,2,\cdots,n$ and $x\in \g$. Here, the notation $[\cdot,\cdot]_F$ is defined to be
$$[\huaX,\huaY]_F:=\Courant{x_1,x_2,y_1}_\g\ot y_2+y_1\ot\Courant{x_1,x_2,y_2}_\g,$$
for any $\huaX=x_1\ot x_2$ and $\huaY=y_1\ot y_2$. Yamaguti showed that $\de^2=0$ and thus $(C^*_{\mathsf{Ymg}}(\g;V):=\oplus_{n=1}^{\infty}C^n_{\mathsf{Ymg}}(\g;V),\de)$ is a cochain complex, whose cohomology is called the {\bf Yamaguti cohomology} \cite{Yamaguti3}.

\vspace{3mm}
Let $\g$ be a vector space and $C^*_{\mathsf{N}}(\g,\g)=\oplus_{p\geqslant0}C^p_{\mathsf{N}}(\g,\g)$, where $C^p_{\mathsf{N}}(\g,\g)=\Hom(\otimes^{2p+1}\g,\g)$ and degrees of elements in $C^p_{\mathsf{N}}(\g,\g)$ are assumed to be $p$. For all $P\in C^p_{\mathsf{N}}(\g,\g)$ and $Q\in C^q_{\mathsf{N}}(\g,\g)$, define a graded bracket (called the {\bf Nambu bracket}) $[\cdot,\cdot]_{\mathsf{N}}$
to be
$$[P,Q]_{\mathsf{N}}=P\circ Q-(-1)^{pq}Q\circ P,$$
where $P\circ Q\in C^{p+q}_{\mathsf{N}}(\g,\g)$ is defined by
{\footnotesize
\begin{eqnarray}
~ &&\nonumber P\circ Q(\huaX_1,\cdots,\huaX_{p+q},x)\\
~ \nonumber&=&\sum_{k=1}^p (-1)^{(k-1)q}\sum_{\sigma\in \mathbb{S}(k-1,q)}(-1)^\sigma P\Big(\huaX_{\sigma(1)},\cdots,\huaX_{\sigma(k-1)},Q\big(\huaX_{\sigma(k)},\cdots,\huaX_{\sigma(k+q-1)},x_{k+q}\big)\otimes y_{k+q}, \huaX_{k+q+1},\cdots,\huaX_{p+q},x\Big)\\
~ &&+\sum_{k=1}^p (-1)^{(k-1)q}\sum_{\sigma\in \mathbb{S}(k-1,q)}(-1)^\sigma P\Big(\huaX_{\sigma(1)},\cdots,\huaX_{\sigma(k-1)},x_{k+q}\otimes Q\big(\huaX_{\sigma(k)},\cdots,\huaX_{\sigma(k+q-1)},y_{k+q}\big), \huaX_{k+q+1},\cdots,\huaX_{p+q},x\Big)\label{control}\\
~ &&\nonumber+\sum_{\sigma\in \mathbb{S}(p,q)}(-1)^\sigma P\Big(\huaX_{\sigma(1)},\cdots,\huaX_{\sigma(p)},Q\big(\huaX_{\sigma(p+1)},\cdots,\huaX_{\sigma(p+q)},x\big)\Big),
\end{eqnarray}}
for all $\huaX_i\in \otimes^2\g,~i=1,2,\cdots,p+q$ and $x\in \g$.

Then $(C^*_{\mathsf{N}}(\g,\g),[\cdot,\cdot]_{\mathsf{N}})$ is a graded Lie algebra and its Maurer-Cartan elements corresponds to Nambu algebra structures \cite{Rot}. Let $C^*_{\mathsf{LTS}}(\g,\g)=\oplus_{p\geqslant0}C_{\mathsf{LTS}}^p(\g,\g)=\oplus_{p\geqslant0}\Hom_{\mathsf{LTS}}^p(\ot^{2p+1}\g,\g)$ be a graded subspace of $C^*_{\mathsf{N}}(\g,\g)$ such that for any $P\in C^p_{\mathsf{LTS}}(\g,\g)$, $P$ satisfies
\begin{eqnarray*}
P(\huaX_1,\cdots,\huaX_{p-1},x,x,y)&=&0,\\
~P(\huaX_1,\cdots,\huaX_{p-1},x,y,z)+P(\huaX_1,\cdots,\huaX_{p-1},y,z,x)+P(\huaX_1,\cdots,\huaX_{p-1},z,x,y)&=&0.
\end{eqnarray*}
The corresponding Nambu bracket is denoted by $[\cdot,\cdot]_{\mathsf{LTS}}$, when the graded vector space is restricted  to $C^*_{\mathsf{LTS}}(\g,\g)=\oplus_{p\geqslant0}C_{\mathsf{LTS}}^p(\g,\g)$. Then $(C^*_{\mathsf{LTS}}(\g,\g)=\oplus_{p\geqslant0}C_{\mathsf{LTS}}^p(\g,\g),[\cdot,\cdot]_{\mathsf{LTS}})$ is a graded subalgebra and its Maurer-Cartan elements corresponds to Lie triple system structures \cite{XST}. In the present paper, we always consider the graded Lie algebra $(C^*_{\mathsf{LTS}}(\g,\g),[\cdot,\cdot]_{\mathsf{LTS}})$.

Consider a Lie triple system $(\g,\Courant{\cdot,\cdot,\cdot}_\g)$ and set $\pi(x,y,z)=\Courant{x,y,z}_\g$. Let $\de:C^n_{\mathsf{Ymg}}(\g;\g)\longrightarrow C^{n+1}_{\mathsf{Ymg}}(\g;\g)$ be the coboundary map associated to the adjoint representation. Then we have the following
\begin{pro}(\cite{XST})
For any $f\in C^n_{\mathsf{Ymg}}(\g,\g)$, we have
$$\de (f)=(-1)^{n-1}[\pi,f]_{\mathsf{LTS}}.$$
\end{pro}

\subsection{Quasi-twilled Lie triple systems}
In this section, let us first recall some notions in \cite{ZXQ}. Let $\g_1$ and $\g_2$ be vector spaces, and elements in $\g_1$ are denoted by $x,y,z,x_i$, and those in $\g_2$ by $u,v,w,u_i$. We denote by $\g^{l,k}$ the direct sum of all $(l+k)$-tensor power of $\g_1$ and $\g_2$: $\otimes^{n}(\g_1\oplus\g_2)$, where $n=l+k$, and $l$ (resp. $k$) means the quantities of $\g_1$ (resp. $\g_2$). For example, $\g^{2,1}\subset \ot^3(\g_1\oplus\g_2)$ can be written as
$$\g^{2,1}=(\g_1\ot\g_1\ot\g_2)\oplus(\g_1\ot\g_2\ot\g_1)\oplus(\g_2\ot\g_1\ot\g_1).$$
Then $\displaystyle{\otimes^{n}(\g_1\oplus\g_2)=\oplus_{l+k=n}\g^{l,k}}$. For example,
$$\ot^3(\g_1\oplus\g_2)=\g^{3,0}\oplus\g^{2,1}\oplus\g^{1,2}\oplus\g^{0,3}.$$
By $\Hom$-functor, we have
\begin{eqnarray}
C^p_{\mathsf{LTS}}(\g_1\oplus\g_2,\g_1\oplus\g_2)=\bigoplus_{l+k=2p}C^p_{\mathsf{LTS}}(\g^{l,k},\g_1)\oplus\bigoplus_{l+k=2p}C^p_{\mathsf{LTS}}(\g^{l,k},\g_2).\label{decompose}
\end{eqnarray}
For a linear map $f\in C^p_{\mathsf{LTS}}(\g^{l,k},\g_1)$ (resp. $f\in C^p_{\mathsf{LTS}}(\g^{l,k},\g_2)$), $f$ naturally induces a linear map $\hat{f}\in C^p_{\mathsf{LTS}}(\g_1\oplus\g_2,\g_1\oplus\g_2)$ (called the {\bf lift} of $f$) defined to be
\begin{eqnarray*}
\hat f:=
\begin{cases}
f,\quad \text{on }~~ \g^{l,k},\\
0,\quad \text{all other cases.}
\end{cases}
\end{eqnarray*}

\begin{defi}\cite{ZXQ}
Let $(\huaG,\Courant{\cdot,\cdot,\cdot}_{\huaG})$ be a Lie triple system with a decomposition of two subspaces: $\huaG=\g_1\oplus\g_2$. We call the triple $(\huaG,\g_1,\g_2)$ a {\bf quasi-twilled Lie triple system} if $\g_2$ is a subalgebra.
\end{defi}

The operation $\Courant{(x,u),(y,v),(z,w)}_\huaG$ of $\huaG$ is uniquely decomposed into 12 multiplications by the canonical projections $\huaG\longrightarrow\g_1$ and $\huaG\longrightarrow\g_2$:
\begin{eqnarray*}
\Courant{x,y,z}_\huaG&=&(\Courant{x,y,z}_1,\Courant{x,y,z}_2), \quad \Courant{x,y,w}_\huaG=(\Courant{x,y,w}_1,\Courant{x,y,w}_2),\\
~\Courant{x,v,z}_\huaG&=&(\Courant{x,v,z}_1,\Courant{x,v,z}_2), \quad \Courant{x,v,w}_\huaG=(\Courant{x,v,w}_1,\Courant{x,v,w}_2),\\
~\Courant{u,v,z}_\huaG&=&(\Courant{u,v,z}_1,\Courant{u,v,z}_2),  \quad \Courant{u,v,w}_\huaG=(\Courant{u,v,w}_1,\Courant{u,v,z}_2).
\end{eqnarray*}
Here the operation $\Courant{\cdot,\cdot,\cdot}_1$ (resp. $\Courant{\cdot,\cdot,\cdot}_2$) denotes the projection of $\huaG$ onto $\g_1$ (resp. $\g_2$).

Let $(\huaG,\g_1,\g_2)$ be a quasi-twilled Lie triple system. We use the notation $\Theta$ to denote the operation $\Courant{\cdot,\cdot,\cdot}_\huaG$ on $\huaG$, i.e.,
$$\Theta((x,u),(y,v),(z,w))=\Courant{(x,u),(y,v),(z,w)}_\huaG.$$

In general, Let the triple $(\huaG,\g_1,\g_2)$ be a quasi-twilled Lie triple system. Write the structure as $\Theta=\hat\phi_1+\hat\mu_1+\hat\psi+\hat\mu_2$, then $\hat\phi_1,~\hat\mu_1,~\hat\psi$ and $\hat\mu_2$ have forms as follows
{\footnotesize
\begin{eqnarray}\label{equia}
\left\{\begin{array}{rcl}
~~\hat\phi_1((x,u),(y,v),(z,w))&=&(0,\Courant{x,y,z}_2),\\
~~\hat\mu_1((x,u),(y,v),(z,w))&=&(\Courant{x,y,z}_1,\Courant{x,y,w}_2+\Courant{u,y,z}_2-\Courant{v,x,z}_2),\\
~~\hat\psi((x,u),(y,v),(z,w))&=&(\Courant{x,y,w}_1+\Courant{u,y,z}_1-\Courant{v,x,z}_1,\Courant{u,v,z}_2+\Courant{x,v,w}_2-\Courant{y,u,w}_2),\\
~~\hat\mu_2((x,u),(y,v),(z,w))&=&(\Courant{u,v,z}_1+\Courant{x,v,w}_1-\Courant{y,u,w}_1,\Courant{u,v,w}_2).
\end{array}\right.
\end{eqnarray}}
It is easy to see that $\hat\phi_1$ is the lift of $\phi_1:\g_1\longrightarrow\g_2$, where $\phi_1=\Courant{\cdot,\cdot,\cdot}_2$. If $\hat\phi_1=0$, we call the triple $(\huaG,\g_1,\g_2  )$ a {\bf twilled Lie triple system}. If $\hat\psi=0$, then we call the triple $(\huaG,\g_1,\g_2  )$ a {\bf strict quasi-twilled Lie triple system}. If moreover, $\hat\phi_1=0$ and $\hat\psi=0$, then we call the triple $(\huaG,\g_1,\g_2  )$ a {\bf strict twilled Lie triple system}, in which case, Lie triple systems $(\g_1,\Courant{\cdot,\cdot,\cdot}_1)$ and $(\g_2,\Courant{\cdot,\cdot,\cdot}_2)$ form a matched pair of Lie triple systems \cite{ZXQ}. In this paper, we focus on the strict case, and we always omit the premix ``strict" without ambiguities.
\emptycomment{
In fact, $[\Theta,\Theta]_{\mathsf{LTS}}=0$ is equivalent to the following conditions:
\begin{eqnarray*}
\left\{\begin{array}{rcl}
[\hat\phi_1,\hat\mu_1]_{\mathsf{LTS}}&=&0,\\
~[\hat\psi,\hat\phi_1]_{\mathsf{LTS}}+\half[\hat\mu_1,\hat\mu_1]_{\mathsf{LTS}}&=&0,\\
~[\hat\phi_1,\hat\mu_2]_{\mathsf{LTS}}+[\hat\psi,\hat\mu_1]_{\mathsf{LTS}}&=&0,\\
~[\hat\mu_1,\hat\mu_2]_{\mathsf{LTS}}+\half[\hat\psi,\hat\psi]_{\mathsf{LTS}}&=&0,\\
~[\hat\psi,\hat\mu_2]_{\mathsf{LTS}}&=&0,\\
~\half[\hat\mu_2,\hat\mu_2]_{\mathsf{LTS}}&=&0.
\end{array}\right.
\end{eqnarray*}}

In the sequel, we give some examples of quasi-twilled Lie triple systems.

\begin{defi}
Let $(\g,\Courant{\cdot,\cdot,\cdot}_\g)$ and $(\h,\Courant{\cdot,\cdot,\cdot}_\h)$ be two Lie triple systems. If $(\h;\rho)$ is a representation of $\g$ and the linear map $\rho:\ot^2\g\longrightarrow\gl(\h)$ satisfies the following conditions
\begin{eqnarray*}
\rho(x,y)u&\in& C(\h),\\
\rho(x,y)\Courant{u,v,w}_\h&=&0,\quad \forall x,y\in \g,~u,v,w\in \h,
\end{eqnarray*}
then we call $\rho$ an {\bf action} of $\g$ on $\h$, where $C(\h)$ means the center of $\h$:
$$C(\h):=\{u\in \h:\Courant{u,v,w}=0,\forall v,w\in \h\}\cap\{u\in \h:\Courant{v,w,u}=0,\forall v,w\in \h\}.$$
\end{defi}

\begin{ex}\label{crossedd}
Let $(\g,\Courant{\cdot,\cdot,\cdot}_\g)$ and $(\h,\Courant{\cdot,\cdot,\cdot}_\h)$ be two Lie triple systems, and $\rho:\ot^2\g\longrightarrow\gl(\h)$ an action of $\g$ on $\h$. For any $\lambda\in \mathbb K$ and for all $x,y,z\in \g,~u,v,w\in \h$, define an operation $\Courant{\cdot,\cdot,\cdot}_\rho$ on $\g\oplus\h$ to be
$$\Courant{(x,u),(y,v),(z,w)}_\rho:=(\Courant{x,y,z}_\g,D_\rho(x,y)w+\rho(y,z)u-\rho(x,z)v+\lambda\Courant{u,v,w}_\h).$$
Then $(\g\oplus\h,\Courant{\cdot,\cdot,\cdot}_\rho)$ is a Lie triple system, called the {\bf action Lie triple system}, which is denoted by $\g\ltimes_\rho\h$. Consequently, $(\g\ltimes_\rho\h,\g,\h)$ becomes a (quasi-)twilled Lie triple system, where $\Courant{\cdot,\cdot,\cdot}_\g:=\Courant{\cdot,\cdot,\cdot}_1$, $\rho(x,y)v:=\Courant{v,x,y}_2$ and $\lambda\Courant{\cdot,\cdot,\cdot}_\h:=\Courant{\cdot,\cdot,\cdot}_2$.
\end{ex}

\begin{ex}\label{semidirectt}
Let $(\g,\Courant{\cdot,\cdot,\cdot}_\g)$ be a Lie triple system, and $(V;\rho)$ a representation of $(\g,\Courant{\cdot,\cdot,\cdot}_\g)$. Then the {\bf semidirect product Lie triple system} $\g \ltimes_{\rho} V$, whose operation $\Courant{\cdot,\cdot,\cdot}_{\rho}$ on $\g \oplus V$ is defined to be for all $x,y,z \in \g, ~u,v,w \in V$,
\begin{eqnarray*}
~\Courant{(x,u),(y,v),(z,w)}_{\rho}=(\Courant{x,y,z}_\g,D_\rho(x,y)w+\rho(y,z)u-\rho(x,z)v),
\end{eqnarray*}
forms a quasi-twilled Lie triple system $(\g\ltimes_\rho V,\g,V)$. In this case, $\Courant{\cdot,\cdot,\cdot}_\g=\Courant{\cdot,\cdot,\cdot}_1$ and $\rho(x,y)v:=\Courant{v,x,y}_2$.
\end{ex}

\begin{ex}\label{directt}
Let $(\g,\Courant{\cdot,\cdot,\cdot}_\g)$ and $(\h,\Courant{\cdot,\cdot,\cdot}_\h)$ be two Lie triple systems. Then the direct product Lie triple system $(\g\oplus\h,\Courant{\cdot,\cdot,\cdot}_\oplus)$, where the operation $\Courant{\cdot,\cdot,\cdot}_{\g\oplus\h}$ is defined to be for all $x,y,z\in \g,~u,v,w\in \h$
$$\Courant{(x,u),(y,v),(z,w)}_{\g\oplus\h}:=(\Courant{x,y,z}_\g,\Courant{u,v,w}_\h),$$
forms a quasi-twilled Lie triple system $(\g\oplus\h,\g,\h)$. In this case, $\Courant{\cdot,\cdot,\cdot}_\g:=\Courant{\cdot,\cdot,\cdot}_1$ and $\Courant{\cdot,\cdot,\cdot}_\h:=\Courant{\cdot,\cdot,\cdot}_2$.
\end{ex}

\begin{ex}\label{twisted}
Let $(\g,\Courant{\cdot,\cdot,\cdot}_\g)$ be a Lie triple system, $(V;\rho)$ a representation of $(\g,\Courant{\cdot,\cdot,\cdot})$, and $\omega\in \Hom(\wedge^2\g\ot\g,V)$ a $1$-cocycle. Then $(\g\oplus V,\Courant{\cdot,\cdot,\cdot}_{\rho,\omega})$ is a Lie triple system, where  the operation $\Courant{\cdot,\cdot,\cdot}_{\rho,\omega}$ is defined to be for all $x,y,z \in \g, ~u,v,w \in V$
$$\Courant{(x,u),(y,v),(z,w)}_{\rho,\omega}:=(\Courant{x,y,z}_\g,D_\rho(x,y)w+\rho(y,z)u-\rho(x,z)v+\omega(x,y,z)).$$
Denote this Lie triple system by $\g\ltimes_{\rho,\omega}V$, and thus $(\g\ltimes_{\rho,\omega}V,\g,V)$ becomes a quasi-twilled Lie triple system, where $\Courant{\cdot,\cdot,\cdot}_\g:=\Courant{\cdot,\cdot,\cdot}_1$, $\rho(x,y)u:=\Courant{u,x,y}_2$ and $\omega(B(u),B(v),B(w)):=\Courant{u,v,w}_2$.
\end{ex}

\begin{ex}\label{Rey}
As a special case of Example \ref{twisted}, consider $V=\g$, $\rho=\frkR$ the adjoint representation as in Example \ref{ad}, and $\omega=\Courant{\cdot,\cdot,\cdot}_\g$. Then we obtain a (quasi-)twilled Lie triple system $(\g\ltimes_{\frkR,\omega}\g,\g,\g)$, where $\Courant{\cdot,\cdot,\cdot}_\g:=\Courant{\cdot,\cdot,\cdot}_1$ and $\rho(x,y)u:=\Courant{u,x,y}_2$.
\end{ex}

Before giving another example, let us recall the notion of matched pairs of Lie triple systems introduced in \cite{ZXQ}.

\begin{defi}
A {\bf matched pair} of Lie triple systems consists of a pair of two Lie triple systems $(\g_1,\Courant{\cdot,\cdot,\cdot}_{\g_1})$ and $(\g_2,\Courant{\cdot,\cdot,\cdot}_{\g_2})$, and a pair of two linear maps $\rho_1:\ot^2\g_1\longrightarrow\gl(\g_2)$ and $\rho_2:\ot^2\g_2\longrightarrow\gl(\g_1)$, such that the following conditions are satisfied:
\begin{itemize}
\item[\rm(i)] $(\g_2;\rho_1)$ is a representation of $\g_1$;
\item[\rm (ii)] $(\g_1;\rho_2)$ is a representation of $\g_2$;
\item[\rm (iii)] For all $x,y,z\in \g_1$ and $u,v,w\in \g_2$, the following equalities hold:
\begin{eqnarray*}
~ \rho_2(u,v)\Courant{x,y,z}_{\g_1}&=&\Courant{x,y,\rho_2(u,v)z}_{\g_1}-\rho_2(D_1(x,y)u,v)z-\rho_2(u,D_1(x,y)v)z,\\
~\label{mtp8} \Courant{x,y,\rho_2(u,v)z}_{\g_1}&=&\rho_2(u,D_1(x,y)v)z+\rho_2(\rho_1(z,y)u,v)x-\rho_2(\rho_1(z,x)u,v)y,\\
~\label{mtp10} \Courant{\rho_2(u,v)x,y,z}_{\g_1}&=&\rho_2(u,\rho_1(y,z)v)x+D_2(v,\rho_1(x,y)u)z-\rho_2(v,\rho_1(x,z)u)y,\\
~ \rho_1(x,y)\Courant{u,v,w}_{\g_2}&=&\Courant{u,v,\rho_1(x,y)w}_{\g_2}-\rho_1(D_2(u,v)x,y)w-\rho_1(x,D_2(u,v)y)w\\
~\label{mtp16}  \Courant{u,v,\rho_1(x,y)w}_{\g_2}&=&\rho_1(x,D_2(u,v)y)w+\rho_1(\rho_2(w,v)x,y)u-\rho_1(\rho_2(w,u)x,y)v,\\
~\label{mtp18} \Courant{\rho_1(x,y)u,v,w}_{\g_2}&=&\rho_1(x,\rho_2(v,w)y)u+D_1(y,\rho_2(u,v)x)y-\rho_1(y,\rho_2(u,w)x)v,
\end{eqnarray*}
\end{itemize}
where $D_1(x,y)=D_{\rho_1}(x,y)=\rho_1(y,x)-\rho_1(x,y)$ and $D_2(u,v)=D_{\rho_2}(u,v)=\rho_2(v,u)-\rho_2(u,v)$. We denote a matched pair of Lie triple systems by a quadruple $(\g_1,\g_2;\rho_1,\rho_2)$ or $\Big((\g_1,\rho_2),(\g_2,\rho_1)\Big)$.
\end{defi}

Let $(\g_1,\g_2;\rho_1,\rho_2)$ be a matched pair of Lie triple systems, then there exists a Lie triple system structure $\Courant{\cdot,\cdot,\cdot}_{\bowtie}$ on the direct sum $\g_1\oplus\g_2$ given by for all $x,y,z\in \g_1,u,v,w\in \g_2$,
\begin{eqnarray*}
~ &&\Courant{(x,u),(y,v),(z,w)}_{\bowtie}\\
~ &=&(\Courant{x,y,z}_{\g_1}+D_2(u,v)z+\rho_2(v,w)z-\rho_2(u,w)y,\Courant{u,v,w}_{\g_2}+D_1(x,y)w+\rho_2(y,z)u-\rho_2(x,z)v).
\end{eqnarray*}
We call this Lie triple system the {\bf double} of $(\g_1,\g_2;\rho_1,\rho_2)$ and denote it by $(\g_1\bowtie\g_2,\Courant{\cdot,\cdot,\cdot}_{\bowtie})$.

\begin{ex}\label{defor}
Let $(\g_1,\g_2;\rho_1,\rho_2)$ be a matched pair of Lie triple system, and $(\g_1\bowtie\g_2,\Courant{\cdot,\cdot,\cdot}_{\bowtie})$ its double. Then $(\g_1\bowtie\g_2,\g_1,\g_2)$ forms a quasi-twilled Lie triple system, in which case $\Courant{x,y,z}_{\g_1}:=\Courant{x,y,z}_1, D_2(u,v)z:=\Courant{u,v,z}_1,\rho_2(v,w)x:=\Courant{x,v,w}_1,\Courant{u,v,w}_{\g_2}:=\Courant{u,v,w}_2, D_1(u,v)z:=\Courant{x,y,w}_2$ and $\rho_1(y,z)u:=\Courant{u,y,z}_2$.
\end{ex}

\section{Cohomology for deformation maps of type I}
Recall that if the triple $(\huaG,\g_1,\g_2)$ is a quasi-twilled Lie triple system, then we write the operation on $\huaG$ as $\Theta=\hat\phi_1+\hat\mu_1+\hat\mu_2$, where $\hat\phi_1,~\hat\mu_1, ~\hat\mu_2$ have the form
\begin{eqnarray}\label{equi}
\left\{\begin{array}{rcl}
~~\hat\phi_1((x,u),(y,v),(z,w))&=&(0,\Courant{x,y,z}_2),\\
~~\hat\mu_1((x,u),(y,v),(z,w))&=&(\Courant{x,y,z}_1,\Courant{x,y,w}_2+\Courant{u,y,z}_2-\Courant{v,x,z}_2),\\
~~\hat\mu_2((x,u),(y,v),(z,w))&=&\Big(\Courant{u,v,z}_1+\Courant{x,v,w}_1-\Courant{y,u,w}_1,\Courant{u,v,w}_2\Big).
\end{array}\right.
\end{eqnarray}

\subsection{Deformation maps of type I of quasi-twilled Lie triple systems}
\begin{defi}
Let the triple $(\huaG,\g_1,\g_2)$ be a quasi-twilled Lie triple system. A {\bf deformation map of type I} ($\huaD$-map for short) of $(\huaG,\g_1,\g_2)$ is a linear map $D:\g_1\longrightarrow\g_2$ such that for any $x,y,z\in \g_1$,
\begin{eqnarray*}
~ &&D\Big(\Courant{x,y,z}_1+\Courant{D(x),D(y),z}_1+\Courant{x,D(y),D(z)}_1-\Courant{y,D(x),D(z)}_1\Big)\\
~ &=&\Courant{x,y,z}_2+\Courant{x,y,D(z)}_2+\Courant{D(x),y,z}_2-\Courant{D(y),x,z}_2+\Courant{D(x),D(y),D(z)}_2.
\end{eqnarray*}
\end{defi}

\begin{rmk}
Note that $\huaD$-maps may not exist. For example, let us consider the quasi-twilled Lie triple system given in Example \ref{twisted}. In this case, the linear map $D:\g\longrightarrow V$ is a $\huaD$-map if and only if the following equation holds
$$\omega(x,y,z)=D(\Courant{x,y,z}_\g)-D_\rho(x,y)D(z)-\rho(y,z)D(x)+\rho(x,z)D(y)=\de(-D)(x,y,z),\quad\forall x,y,z\in \g,$$
where $\de$ is the coboundary map of the Lie triple system $(\g,\Courant{\cdot,\cdot,\cdot}_\g)$ with coefficients in the representation $(V;\rho)$, which implies that $\omega$ is a 2-coboundary. ($\omega$ is also a 1-cocycle).
\end{rmk}

Let $D:\g_1\longrightarrow\g_2$ be a linear map. Denote the graph of $D$ by
$$Gr(D):=\{(x,D(x))|x\in \g_1\}.$$

As we can see in the following proposition, the condition of $\huaD$-map is sufficient and necessary in making the graph of $D$ being a subalgebra.

\begin{pro}\label{graph}
Let $(\huaG,\g_1,\g_2)$ be a twilled Lie triple system. A linear map $D:\g_1\longrightarrow\g_2$ is a $\huaD$-map if and only if $Gr(D)$ is a subalgebra of $\huaG$. In this case, $(\g_2,Gr(D))$ is a matched pair of Lie triple systems.
\end{pro}
\begin{proof}
For all $(x,D(x)),(y,D(y)),(z,D(z))\in Gr(D)$, we have
\begin{eqnarray*}
~ &&\Theta\Big((x,D(x)),(y,D(y)),(z,D(z))\Big)\\
~ &=&\Big(\Courant{x,y,z}_1+\Courant{D(x),D(y),z}_1+\Courant{x,D(y),D(z)}_1-\Courant{y,D(x),D(z)}_1,\\
~ &&\Courant{x,y,z}_2+\Courant{x,y,D(z)}_2+\Courant{D(x),y,z}_2-\Courant{D(y),x,z}_1+\Courant{D(x),D(y),D(z)}_2\Big).
\end{eqnarray*}
Thus, $Gr(D)$ is a subalgebra of $\huaG$, i.e., $\Theta\Big((x,D(x)),(y,D(y)),(z,D(z))\Big)\in Gr(D)$, if and only if
\begin{eqnarray*}
~ &&D\Big(\Courant{x,y,z}_1+\Courant{D(x),D(y),z}_1+\Courant{x,D(y),D(z)}_1-\Courant{y,D(x),D(z)}_1\Big)\\
~ &=&\Courant{x,y,z}_2+\Courant{x,y,D(z)}_2+\Courant{D(x),y,z}_2-\Courant{D(y),x,z}_1+\Courant{D(x),D(y),D(z)}_2,
\end{eqnarray*}
which implies that $D$ is a $\huaD$-map of $(\huaG,\g_1,\g_2)$.

Since $\huaG=Gr(D)\oplus\g_2$, it follows that $(\g_2,Gr(D))$ is a matched pair of Lie tripe systems.
\end{proof}

We give several examples of $\huaD$-maps on each special quasi-twilled Lie triple systems as in Examples \ref{crossedd}-\ref{directt}.
\begin{ex}\label{crossed}
Consider the quasi-twilled Lie triple system $(\g\ltimes_\rho\h,\g,\h)$ given in Example \ref{crossedd}. In this case, a $\huaD$-map is a linear map $D:\g\longrightarrow\h$ such that for any $x,y,z\in \g$,
$$D\Courant{x,y,z}=D_\rho(x,y)D(z)+\rho(y,z)D(x)-\rho(x,z)D(y)+\lambda\Courant{D(x),D(y),D(z)},$$
which is exactly the {\bf crossed homomorphism of weight $\lambda$} from $\g$ to $\h$.
\end{ex}

\begin{ex}\label{semidirect}
Consider the quasi-twilled Lie triple system $(\g\ltimes_\rho V,\g,V)$ given in Example \ref{semidirectt}. In this case, a $\huaD$-map is a linear map $D:\g\longrightarrow V$ such that for all $x,y,z\in \g$,
$$D\Courant{x,y,z}=D_\rho(x,y)D(z)+\rho(y,z)D(x)-\rho(x,z)D(y),$$
which implies that $D$ is a {\bf derivation} from $\g$ to $V$. In particular, if we take $\rho=\frkR$ to be the adjoint representation of $\g$ on itself, then we obtain the usual derivation.
\end{ex}

\begin{ex}\label{direct}
Consider the quasi-twilled Lie triple system $(\g\oplus\h,\g,\h)$ given in Example \ref{directt}. In this case, a $\huaD$-map is a linear map $D:\g\longrightarrow\h$ such that
$$D(\Courant{x,y,z}_\g)=\Courant{D(x),D(y),D(z)}_\h,\quad \forall x,y,z\in \g,$$
which is exactly a {\bf homomorphism} from $\g$ to $\h$.
\end{ex}

At the end of this subsection, we show that a $\huaD$-map plays a crucial role in the twisting theory. We first recall the notion of twisting \cite{ZXQ}. Before this, let us introduce a notation. Suppose that $(\huaG,\g_1,\g_2)$ is a quasi-twilled Lie triple system, and $f\in C^1_{\mathsf{LTS}}(\huaG,\huaG)$ is a $1$-cochain. Set
$$\exp(X_f)(\cdot):=\sum_{k=0}^\infty \frac{1}{k!}X_f^k,$$
where $X_f^k$ is defined to be
\begin{eqnarray*}
X_f^k:=
\begin{cases}
\underbrace{[\cdots[[}_k\cdot,f]_{\mathsf{LTS}},f]_{\mathsf{LTS}},\cdots,f]_{\mathsf{LTS}}, & k\geq 1,\\
\Id, & k=0.
\end{cases}
\end{eqnarray*}
Note that $\exp(X_f)(\cdot)$ is not well defined in general. If $D:\g_1\longrightarrow\g_2$ is a linear map, then it follows that its lift is $\hat D\in \Hom(\huaG,\huaG)$ defined to be
$$\hat D(x,u)=(0,D(x)),\quad \forall x\in \g_1,~u\in \g_2,$$
and that $\hat D\circ\hat D=0$.

\begin{defi}
Let $((\huaG,\Theta),\g_1,\g_2)$ be a quasi-twilled Lie triple system and $D:\g_1\longrightarrow\g_2$ a linear map. The transformation $\Theta^D:=\exp(X_{\hat D})(\Theta)$ is called the {\bf twisting} of $\Theta$ by $D$.
\end{defi}

The twisting of $\Theta$ has the following properties, and one can see \cite{ZXQ} for more details.
\begin{lem}\label{TTTtwisting}
With the above notations, we have the following conclusions:
\begin{itemize}
\item[\rm(i)] the following identity holds
\begin{eqnarray}
\Theta^D=\exp(-\hat D)\circ \Theta\circ \Big(\exp(\hat D)\ot\exp(\hat D)\ot\exp(\hat D)\Big);\label{tw}
\end{eqnarray}
\item[\rm(ii)] the twisting $\Theta^D$ is a Lie triple system structure on $\huaG$;
\item[\rm(iii)] the map $\exp(\hat D):(\huaG,\Theta^D)\longrightarrow (\huaG,\Theta)$ is an isomorphism between Lie triple systems.
\end{itemize}
\end{lem}

\begin{pro}\label{strictquasi}
Let $((\huaG,\Theta),\g_1,\g_2)$ be a twilled Lie triple system and $D:\g_1\longrightarrow\g_2$ a linear map. Then the twisting $((\huaG,\Theta^D),\g_1,\g_2)$ is still a (strict) quasi-twilled Lie triple system.
\end{pro}
\begin{proof}
By \eqref{tw}, we have for all $u,v,w\in \g_2$,
\begin{eqnarray*}
\Theta^D\Big((0,u),(0,v),(0,w)\Big)&=&\exp(-\hat D)\circ\Theta\circ\Big(\exp(\hat D)(0,u)\ot\exp(\hat D)(0,v)\ot\exp(\hat D)(0,w)\Big)\\
~ &=&\Theta\Big((0,u),(0,v),(0,w)\Big)\in \g_2,
\end{eqnarray*}
which implies that $\g_2$ is a subalgebra of $(\huaG,\Theta^D)$. Thus $((\huaG,\Theta^D),\g_1,\g_2)$ is a quasi-twilled Lie triple system.
\end{proof}

Write $\Theta^D:=\hat\mu_2^D+\hat\mu_1^D+\hat\phi_1^D$, then the substructures $\hat\mu_2^D,~\hat\mu_1^D,~\hat\phi_1^D$ have the form as follows:
\begin{eqnarray}\label{strict-quasi}
\left\{\begin{array}{rcl}
\hat\mu_2^D&=&\hat \mu_2,\\
\hat\mu_1^D&=&\hat\mu_1+\half[[\hat\mu_2,\hat D]_{\mathsf{LTS}},\hat D]_{\mathsf{LTS}},\\
\hat\phi_1^D&=&\hat\phi_1+[\hat\mu_1,\hat D]_{\mathsf{LTS}}+\frac{1}{6}[[[\hat\mu_2,\hat D]_{\mathsf{LTS}},\hat D]_{\mathsf{LTS}},\hat D]_{\mathsf{LTS}}.
\end{array}\right.
\end{eqnarray}

The following proposition shows that a $\huaD$-map plays an important role in making a quasi-twilled Lie triple system into a matched pair.
\begin{pro}\label{base}
Let $(\huaG,\g_1,\g_2)$ be a quasi-twilled Lie triple system and $D:\g_1\longrightarrow\g_2$ a linear map. Then Lie triple systems $(\g_1,\Courant{\cdot,\cdot,\cdot}_D)$ and $(\g_2,\Courant{\cdot,\cdot,\cdot}_2)$ form a matched pair of Lie triple systems $\Big(((\g_1,\Courant{\cdot,\cdot,\cdot}_D),\rho_D),((\g_2,\Courant{\cdot,\cdot,\cdot}_2),\rho_2)\Big)$  if and only if $D$ is a $\huaD$-map, where $\Courant{\cdot,\cdot,\cdot}_D$ and $\rho_D:\g_1\longrightarrow\gl(\g_2)$ are given to be for all $x,y\in \g_1,u\in \g_2$
\begin{eqnarray*}
\Courant{x,y,z}_D&:=&\Courant{x,y,z}_1+\Courant{D(x),D(y),z}_1+\Courant{x,D(y),D(z)}_1-\Courant{y,D(x),D(z)}_1,\\
\rho_D(x,y)u&:=&\Courant{u,x,y}_2+\Courant{u,D(x),D(y)}_2-D\bigg(\Courant{u,D(x),y}_1+\Courant{u,y,D(x)}_1\bigg).
\end{eqnarray*}
Here, $\rho_2:\ot^2\g_2\longrightarrow\gl(\g_1)$ is given by
$$\rho_2(u,v)x:=\Courant{x,u,v}_1,\quad \forall u,v\in \g_2,x\in \g_1.$$
\end{pro}
\begin{proof}
It follows from Proposition \ref{strictquasi} that $((\huaG,\Theta^D),\g_1,\g_2)$ is a strict quasi-twilled  Lie triple system and write $\Theta^D=\hat\phi_1^D+\hat\mu_1^D+\hat\mu_2^D$. By \eqref{strict-quasi}, we obtain that
{\footnotesize
\begin{eqnarray*}
\hat\mu_2^D((x,u),(y,v),(z,w))&=&(\Big(\Courant{u,v,z}_1+\Courant{y,v,w}_1-\Courant{x,u,w}_1\Big),\Courant{u,v,w}_2),\\
\hat\mu_1^D((x,u),(y,v),(z,w))&=&(\Courant{x,y,z}_1+(\Courant{D(x),D(y),z}_1+\Courant{x,D(y),D(z)}_1-\Courant{y,D(x),D(z)}_1),\\
~ &&\Courant{u,x,y}_2+\Courant{u,D(x),D(y)}_2-\bigg(\Courant{u,D(x),y}_1+\Courant{u,y,D(x)}_1\bigg)+c.p.),\\
\hat\phi_1^D((x,u),(y,v),(z,w))&=&(0,\Courant{x,y,z}_2+(\Courant{D(x),y,z}_2+c.p.)-D\Courant{x,y,z}_1+\Courant{D(x),D(y),D(z)}_2\\
~ &&-D(\Courant{D(x),D(y),z}_1+c.p.)),
\end{eqnarray*}}
for all $x,y,z\in \g_1,~u,v,w\in \g_2$. Thus it is direct to see that $D$ is a $\huaD$-map if and only if $\hat\phi_1^D=0$, which implies that Lie triple systems $(\g_1,\Courant{\cdot,\cdot,\cdot}_D)$ and $(\g_2,\Courant{\cdot,\cdot,\cdot}_2)$ form a matched pair of Lie triple systems.
\end{proof}

\subsection{Cohomology of deformation maps of type I}
In this subsection, we establish a cohomology of $\huaD$-maps, which unifies the cohomologies of crossed homomorphisms of any weight, derivations and homomorphisms on Lie triple systems.
Before this, we should construct a 0-cocycle first.
\begin{pro}\label{1-cocy}
Let $(\huaG,\g_1,\g_2)$ be a quasi-twilled Lie triple system and $D:\g_1\longrightarrow\g_2$ a $\huaD$-map. Define a map
$$\delta:\wedge^2\g_1\longrightarrow\Hom(\g_1,\g_2)$$
to be
$$\delta(x,y)z=\Courant{D(x),y,z}_2-\Courant{D(y),x,z}_2+\Courant{D(x),D(y),D(z)}_2+\Courant{x,y,z}_2,\quad \forall x,y,z\in \g_1.$$
Then $\delta(x,y)$ is a $1$-cocycle of the Lie triple system $(\g_1,\Courant{\cdot,\cdot,\cdot}_D)$ associated to the representation $(\g_2;\rho_D)$ constructed in Proposition \ref{base}.
\end{pro}
\begin{proof}
For all $x,y,x_1,x_2,x_3\in \g_1$, by a direct computation, we have
{\footnotesize
\begin{eqnarray*}
-\Courant{\delta(x,y)x_1,x_2,x_3}_2+\Courant{\delta(x,y)x_2,x_1,x_3}_2+\Courant{x_1,x_2,\delta(x,y)x_3}_2&=&0,\\
-\Courant{\delta(x,y)x_1,D(x_2),D(x_3)}_2+\Courant{\delta(x,y)x_2,D(x_1),D(x_3)}_2+\Courant{D(x_1),D(x_2),\delta(x,y)x_3}_2&=&0,\\
\Courant{\delta(x,y)x_1,D(x_2),x_3}_1-\Courant{\delta(x,y)x_2,D(x_1),x_3}_1-\Courant{\delta(x,y)x_3,D(x_2),x_1}_1+\Courant{\delta(x,y)x_3,D(x_1),x_2}_1&=&0,\\
\Courant{\delta(x,y)x_1,x_3,D(x_2)}_1-\Courant{\delta(x,y)x_2,x_3,D(x_1)}_1-\Courant{\delta(x,y)x_3,x_1,D(x_2)}_1+\Courant{\delta(x,y)x_3,D(x_2),x_1}_1&=&0,
\end{eqnarray*}}
and since $D$ is a $\huaD$-map, the following equality vanishes:
$$\delta(x,y)(\Courant{x_1,x_2,x_3}_1+\Courant{D(x_1),D(x_2),x_3}_1+c.p.).$$
Thus , we have
\begin{eqnarray*}
&&\de(\delta(x,y))(x_1,x_2,x_3)\\
&=&-\Big(\rho_D(x_2,x_3)\delta(x,y)x_1-\rho_D(x_1,x_3)\delta(x,y)x_2\Big)+D_{\rho_D}(x_1,x_2)\delta(x,y)x_3-\delta(x,y)\Courant{x_1,x_2,x_3}_D\\
~ &=&0
\end{eqnarray*}
This finishes the proof.
\end{proof}

Consequently, we obtain the cohomology of $\huaD$-map of a quasi-twilled Lie triple system. Let $(\huaG,\g_1,\g_2)$ be a quasi-twilled Lie triple system and $D:\g_1\longrightarrow\g_2$ a $\huaD$-map.
Define the space of $n$-cochains $C^n(D)$ to be
\begin{eqnarray*}
C^n(D):=
\begin{cases}
\wedge^2\g_1,& n=0,\\
C^{n}_{\mathsf{Ymg}}(\g_1;\g_2),& n\geq 1.
\end{cases}
\end{eqnarray*}

Define the coboundary map $\delta_D:C^n(D)\longrightarrow C^{n+1}(D)$ to be
\begin{eqnarray*}
\delta_D:=
\begin{cases}
\delta, \quad n=0,\\
\de, \quad n\geq 1.
\end{cases}
\end{eqnarray*}
If $n\geq 1$, the corresponding cochain complex $(\oplus_{n=1}^{+\infty}C^n_{\mathsf{Ymg}}(\g_1;\g_2),\de)$ is that of Lie triple system $(\g_1,\Courant{\cdot,\cdot,\cdot}_D)$ associated to the representation $(\g_2;\rho_D)$, where $\Courant{\cdot,\cdot,\cdot}_D$ and $\rho_D$ are given in Proposition \ref{base}.

From Proposition \ref{1-cocy} and cohomology of Lie triple systems, we obtain the cochian complex $(C^*(D):=\oplus_{n=0}^{+\infty}C^n(D),\delta_D)$, which is defined to be the {\bf cohomology of the $\huaD$-map $D$}. Denote by $Z^n(D)$ the set of $n$-cocycles and by $B^n(D)$ the set of $n$-coboundaries. The $n$-th cohomology group is denoted by
$$H^n(D):=Z^n(D)/B^n(D), \qquad n\geqslant 0.$$

\emptycomment{
In the sequel, we give an intrinsic interpretation of the above coboundary operator.
\begin{pro}
Let $(\huaG,\g_1,\g_2)$ be a quasi-twilled Lie triple system and $D:\g_1\longrightarrow\g_2$ a $\huaD$-map. Then we have
$$l_1^D(f)=(-1)^{k-1}\delta_D(f),$$
where $l_1^D$ is the differential of the twisted $L_\infty$-algebra $(C^*(\g_1,\g_2),l_1^D,l_2^D,l_3^D)$.
\end{pro}
\begin{proof}
For all $f\in C^n(\g_1,\g_2)$, we have
\begin{eqnarray*}
l_1^D(f)&=&l_1(f)+\frac{1}{2}l_3(D,D,f)\\
~ &=&[\hat\mu_1,\hat f]_{\mathsf{LTS}}+\frac{1}{2}[[[\hat\mu_2,\hat D]_{\mathsf{LTS}},\hat D]_{\mathsf{LTS}},\hat f]_{\mathsf{LTS}}\\
~ &=&[\hat\mu_1+\frac{1}{2}[[\hat\mu_2,\hat D]_{\mathsf{LTS}},\hat D]_{\mathsf{LTS}},\hat f]_{\mathsf{LTS}}\\
~ &=&[\hat\mu_1+\hat\mu_2,\hat f]_{\mathsf{LTS}}\\
~ &=&(-1)^{k-1}\de_D(f).
\end{eqnarray*}
This finishes the proof.
\end{proof}
}

In particular, if we restrict the $\huaD$-maps to the cases in Examples \ref{crossed}-\ref{direct} on their corresponding quasi-twilled Lie triple sytems, we obtain cohomologies of crossed homomorphisms of any weight, derivations and homomorphisms on Lie triple systems.

\emptycomment{
\begin{ex}
Consider the quasi-twilled Lie triple system $(\g\ltimes_\rho V,\g,V)$ given in Example \ref{semidirect}. Let $D:\g\longrightarrow V$ be a derivation of from Lie triple system $(\g,\Courant{\cdot,\cdot,\cdot})$ to the $\g$-module $V$.
Define
$$\delta:\wedge^2\g\longrightarrow\Hom(\g,\g)$$
to be
$$\delta(x,y)z:=\Courant{D(x),D(y),D(z)}_\g+\lambda\Courant{x,y,z}_\g.$$
Then the cohomology of modified $r$-matrix is the cochain complex $(C^*(D),\delta)$, where
$$C^n(D)=
\begin{cases}
\wedge^2\g,& n=0,\\
C^{n-1}(\g;\g),& n\geq 1,
\end{cases}$$
and
$$\delta=
\begin{cases}
\delta,& n=0,\\
\de,& n\geq 1.
\end{cases}$$
When $n\geq 1$, the corresponding cochain complex $(C^*(\g,\g),\de)$ is just the Yamaguti cohomology.
\end{ex}
}

\begin{ex}
Consider the quasi-twilled Lie triple system $(\g\ltimes_\rho\h,\g,\h)$ given in Example \ref{crossedd}. Let $D:\g\longrightarrow\h$ be a crossed homomorphism of weight $\lambda$. Then the Lie triple system structure $\Courant{\cdot,\cdot,\cdot}_D:=\Courant{\cdot,\cdot,\cdot}_\g$ and the representation of the Lie triple system $(\g,\Courant{\cdot,\cdot,\cdot})$ on $\h$ is given by
$$\rho_D(x,y)u:=\rho(x,y)u+\lambda\Courant{u,D(x),D(y)}_\h.$$

Define
$$\delta:\wedge^2\g\longrightarrow\Hom(\g,\h)$$
to be
$$\delta(x,y)z:=\rho(y,z)D(x)-\rho(x,z)D(y).$$
Then the cohomology of crossed homomorphism of weight $\lambda$ is the cochain complex $(C^*(D),\delta)$, where
$$C^n(D)=
\begin{cases}
\wedge^2\g,& n=0,\\
C^{n}_{\mathsf{Ymg}}(\g;\h),& n\geq 1,
\end{cases}$$
and
$$\delta=
\begin{cases}
\delta,& n=0,\\
\de,& n\geq 1.
\end{cases}$$
\end{ex}

\begin{ex}
Consider the quasi-twilled Lie triple system $(\g\ltimes_\rho V,\g,V)$ given in Example \ref{semidirectt}. Let $D:\g\longrightarrow V$ be a derivation of from Lie triple system $(\g,\Courant{\cdot,\cdot,\cdot})$ to the $\g$-module $V$. In this case, Then the Lie triple system structure $\Courant{\cdot,\cdot,\cdot}_D:=\Courant{\cdot,\cdot,\cdot}_\g$ and the representation of Lie triple system $(\g,\Courant{\cdot,\cdot,\cdot}_\g)$ on $\h$ is given by $\rho_D:=\rho$.
Define
$$\delta:\wedge^2\g\longrightarrow\Hom(\g,V)$$
to be
$$\delta(x,y)z:=\rho(y,z)D(x)-\rho(x,z)D(y).$$

Then the cohomology of a derivation $D$ from Lie triple system $(\g,\Courant{\cdot,\cdot,\cdot})$ to the $\g$-module $V$ is the cochain complex $(C^*(D),\delta)$, where
$$C^n(D)=
\begin{cases}
\wedge^2\g,& n=0,\\
C^{n}_{\mathsf{Ymg}}(\g;V),& n\geq 1,
\end{cases}$$
and
$$\delta=
\begin{cases}
\delta,& n=0,\\
\de,& n\geq 1.
\end{cases}$$
\end{ex}

\begin{ex}
Consider the quasi-twilled Lie triple system $(\g\oplus\h,\g,\h)$ given in Example \ref{directt}. Let $D:\g\longrightarrow\h$ be a homomorphism. Then the Lie triple system structure $\Courant{\cdot,\cdot,\cdot}_D:=\Courant{\cdot,\cdot,\cdot}_\g$ and the representation of Lie triple system $(\g,\Courant{\cdot,\cdot,\cdot}_\g)$ on $\h$ is given by
$$\rho_D(x,y)u:=\Courant{u,D(x),D(y)}_\h,\quad \forall x,y\in \g,~u\in \h.$$
Define
$$\delta:\wedge^2\g\longrightarrow\Hom(\g,\h)$$
to be
$$\delta(x,y)z:=\Courant{D(x),D(y),D(z)}_\h.$$

Then the cohomology of a homomorphism is the cochain complex $(C^*(D),\delta)$, where
$$C^n(D)=
\begin{cases}
\wedge^2\g,& n=0,\\
C^{n}_{\mathsf{Ymg}}(\g;\h),& n\geq 1,
\end{cases}$$
and
$$\delta=
\begin{cases}
\delta,& n=0,\\
\de,& n\geq 1.
\end{cases}$$
\end{ex}

\section{Cohomology of deformation maps of type II}
In this section, we still focus on the quasi-twilled Lie triple system $(\huaG,\g_1,\g_2)$ whose structure has the form $\Theta:=\hat\phi_1+\hat\mu_1+\hat\mu_2$, where $\hat\phi_1,\hat\mu_1,\hat\mu_2$ are given in \eqref{equi}. More explicitly,
\begin{eqnarray*}
\left\{\begin{array}{rcl}
~~\hat\phi_1((x,u),(y,v),(z,w))&=&(0,\Courant{x,y,z}_2),\\
~~\hat\mu_1((x,u),(y,v),(z,w))&=&(\Courant{x,y,z}_1,\Courant{x,y,w}_2+\Courant{u,y,z}_2-\Courant{v,x,z}_2),\\
~~\hat\mu_2((x,u),(y,v),(z,w))&=&\Big(\Courant{u,v,z}_1+\Courant{x,v,w}_1-\Courant{y,u,w}_1,\Courant{u,v,w}_2\Big).
\end{array}\right.
\end{eqnarray*}

\subsection{Deformation maps of type II of a quasi-twilled Lie triple system}
In this subsection, we introduce the notion of deformation maps of type II on a quasi-twilled Lie triple system, which unifies relative Rota-Baxter operators (of any weight), twisted Rota-Baxter operators and Reynolds operators on Lie triple systems, and deformation maps of matched pairs of Lie triple systems.

\begin{defi}
Let the triple $(\huaG,\g_1,\g_2)$ be a quasi-twilled Lie triple system. A {\bf deformation map of type II} ($\frkD$-map for short) of $(\huaG,\g_1,\g_2)$ is a linear map $B:\g_2\longrightarrow\g_1$ such that for all $u,v,w\in \g_2$,
\begin{eqnarray*}
~ &&\Courant{B(u),B(v),B(w)}_1+\Courant{u,v,B(w)}_1+\Courant{B(u),v,w}_1-\Courant{B(v),u,w}_1\\
~ &=&B\Big(\Courant{B(u),B(v),w}_2+\Courant{u,B(v),B(w)}_2-\Courant{v,B(u),B(w)}_2+\Courant{B(u),B(v),B(w)}_2+\Courant{u,v,w}_2\Big).
\end{eqnarray*}
\end{defi}

These two types of deformation maps are related in the following proposition.
\begin{pro}
Let $D:\g_1\longrightarrow\g_2$ be an invertible map. Then $D$ is a $\huaD$-map if and only if $D^{-1}:\g_2\longrightarrow\g_1$ is a $\frkD$-map.
\end{pro}

Then we provide several $\frkD$-maps on special quasi-twilled Lie triple systems, which are our known operators on Lie triple systems.
\begin{ex}\label{crosseddd}
Consider the quasi-twilled Lie triple system $(\g\ltimes_\rho\h,\g,\h)$ given in Example \ref{crossedd}. A $\frkD$-map of $(\g\ltimes_\rho\h,\g,\h)$ is a linear map $T:\h\longrightarrow\g$ such that
$$\Courant{T(u),T(v),T(w)}_\g=T\Big(D_\rho(T(u),T(v))w+\rho(T(v),T(w))u-\rho(T(u),T(w))v+\lambda\Courant{u,v,w}_\h\Big),~~ \forall u,v,w\in \h,$$
which is exactly a {\bf relative Rota-Baxter operator of weight $\lambda$} on Lie triple system $(\g,\Courant{\cdot,\cdot,\cdot}_\g)$ with respect to an action $\rho$.
\end{ex}

\begin{ex}\label{semidirecttt}
Consider the quasi-twilled Lie triple system $(\g\ltimes_\rho V,\g,V)$ given in Example \ref{semidirectt}. A $\frkD$-map of $(\g\ltimes_\rho V,\g,V)$ is a linear map $T:V\longrightarrow\g$ such that
$$\Courant{T(u),T(v),T(w)}=T\Big(D_\rho(T(u),T(v))w+\rho(T(v),T(w))u-\rho(T(u),T(w))v\Big),\quad \forall u,v,w\in V,$$
which is exactly a {\bf relative Rota-Baxter operator} on $\g$ with respect to the representation $(V;\rho)$.
\end{ex}

\begin{ex}\label{twistedd}
Consider the quasi-twilled Lie triple system $(\g\ltimes_{\rho,\omega}V,\g,V)$ given in Example \ref{twisted}. A $\frkD$-map of $(\g\ltimes_{\rho,\omega}V,\g,V)$ is a linear map $T:V\longrightarrow\g$ such that
$$\Courant{T(u),T(v),T(w)}=T\Big(D_\rho(T(u),T(v))w+\rho(T(v),T(w))u-\rho(T(u),T(w))v+\omega(T(u),T(v),T(w))\Big),$$
for any $u,v,w\in V$, which is exactly a {\bf twisted Rota-Baxter operator}.
\end{ex}

\begin{ex}\label{Reyy}
Consider the quasi-twilled Lie triple system $(\g\ltimes_{\frkR,\omega}\g,\g,\g)$ given in Example \ref{Rey}. A $\frkD$-map of $(\g\ltimes_{\rho,\omega}\g,\g,\g)$ is a linear map $R:\g\longrightarrow\g$ such that
$$\Courant{R(u),R(v),R(w)}=R\Big(D_\rho(R(u),R(v))w+\rho(R(v),R(w))u-\rho(R(u),R(w))v+\Courant{R(u),R(v),R(w)}\Big),$$
for any $u,v,w\in \g$, which is exactly a {\bf Reynolds operator}.
\end{ex}

\begin{ex}\label{deforr}
Let $(\g_1,\g_2;\rho_1,\rho_2)$ be a matched pair of Lie triple system. Consider the quasi-twilled Lie triple system $(\g_1\bowtie\g_2,\g_1,\g_2)$ given in Example \ref{defor}. A $\frkD$-map of $(\g_1\bowtie\g_2,\g_1,\g_2)$ is a linear map $B:\g_2\longrightarrow\g_1$ such that
\begin{eqnarray*}
~ &&\Courant{B(u),B(v),B(w)}_{\g_1}+D_2(u,v)B(w)+\rho_2(v,w)B(u)-\rho_2(u,w)B(v)\\
~ &=&B\Big(\Courant{u,v,w}_{\g_2}+D_1(B(u),B(v))w+\rho_1(B(v),B(w))u-\rho_1(B(u),B(w))v\Big),
\end{eqnarray*}
for all $u,v,w\in \g_2$. In this case, the $\frkD$-map $B$ is called a {\bf deformation map} of matched pair of Lie triple systems.
\end{ex}

In the sequel, we illustrate the roles that $\frkD$-maps play in the twisting theory. Let $B:\g_2\longrightarrow\g_1$ be a linear map. It follows that $\hat B\circ\hat B=0$ and that $[\cdot,\hat B]_{\mathsf{LTS}}$ is a derivation of the graded Lie algebra $(C^*(\huaG,\huaG),[\cdot,\cdot]_{\mathsf{LTS}})$.

\begin{defi}
The transformation $\Theta^B:=\exp(X_{\hat B})(\Theta)$ is a {\bf twisting} of $\Theta$ by $B$.
\end{defi}

Parallel to Lemma \ref{TTTtwisting}, we have the following lemma.

\begin{lem}
With the above notations, $\Theta^B=\exp(-\hat B)\circ\Theta\circ\Big(\exp(\hat B)\ot\exp(\hat B)\ot\exp(\hat B)\Big)$ is a Lie triple system structure on $\huaG$.
\end{lem}

The following proposition shows us that $\frkD$-map plays a crucial role in making the twisting to be a Lie triple system structure.

\begin{pro}\label{twilled}
Let $((\huaG,\Theta),\g_1,\g_2)$ be a quasi-twilled Lie triple system. Then the linear map $B:\g_2\longrightarrow\g_1$ is a $\frkD$-map if and only if the triple $((\huaG,\Theta^B),\g_1,\g_2)$ is a (strict) quasi-twilled Lie triple system.
\end{pro}
\begin{proof}
Write $\Theta=\hat\phi_1+\hat\mu_1+\hat\mu_2$ and $\Theta^B=\hat\phi_1^B+\hat\mu_1^B+\hat\mu_2^B$. Then $\Theta^B$ has the form as follows:
{\footnotesize
\begin{eqnarray*}
\hat\phi_1^B((x,u),(y,v),(z,w))&=&\bigg(-B\Big(\Courant{x+B(u),y+B(v),z+B(w)}_2\Big),\Courant{x+B(u),y+B(v),z+B(w)}_2\bigg),\\
\hat\mu_1^B((x,u),(y,v),(z,w))&=&\bigg(\Courant{x+B(u),y+B(v),z+B(w)}_1-B\Big(\Courant{x+B(u),y+B(v),w}_2+\Courant{u,y+B(v),z+B(w)}_2\\
~ &&-\Courant{v,x+B(u),z+B(w)}_2\Big),\Big(\Courant{x+B(u),y+B(v),w}_2+\Courant{u,y+B(v),z+B(w)}_2-\\
~ && \Courant{v,x+B(u),z+B(w)}_2\Big)\bigg)\\
\hat\mu_2^B((x,u),(y,v),(z,w))&=&\bigg(\Courant{u,v,z+B(w)}_1+\Courant{x+B(u),v,w}_1-\Courant{y+B(v),u,w}_1-B\Courant{u,v,w}_2,\Courant{u,v,w}_2\bigg).
\end{eqnarray*}}
The triple $((\huaG,\Theta^B),\g_1,\g_2)$ is a strict quasi-twilled Lie triple system is equivalent to
$$\Theta^B=\hat\phi_1+\hat\mu_1+\hat\mu_2+[\hat\phi_1,\hat B]_{\mathsf{LTS}}+\half[[\hat\mu_1,\hat B]_{\mathsf{LTS}},\hat B]_{\mathsf{LTS}}+\frac{1}{6}[[[\hat\phi_1,\hat B]_{\mathsf{LTS}},\hat B]_{\mathsf{LTS}},\hat B]_{\mathsf{LTS}}.$$
Comparing terms leads to that $B$ is a $\frkD$-map. This finishes the proof.
\end{proof}

Let the twisting  $((\huaG,\Theta^B),\g_1,\g_2)$ be a strict quasi-twilled Lie triple systems as in Proposition \ref{twilled}, and write its structure $\Theta^B=\hat\phi_1^B+\hat\mu_1^B+\hat\mu_2^B$. Rearranging terms of $\Theta^B$ as follows:
\begin{eqnarray*}
\left\{\begin{array}{rcl}
\hat\phi_1^B&=&\hat\phi_1,\\
\hat\mu_1^B&=&\hat\mu_1+[\hat\phi_1,\hat B]_{\mathsf{LTS}},\\
\hat\mu_2^B&=&\hat\mu_2+\frac{1}{2}[[\hat\mu_1,\hat B]_{\mathsf{LTS}},\hat B]_{\mathsf{LTS}}+\frac{1}{6}[[[\hat\phi_1,\hat B]_{\mathsf{LTS}},\hat B]_{\mathsf{LTS}},\hat B]_{\mathsf{LTS}},
\end{array}\right.
\end{eqnarray*}

\subsection{Cohomology of deformation maps of type II}
In this subsection, we establish a cohomology theory of $\frkD$-maps of a quasi-twilled Lie triple systems by Yamaguti cohomology, and illustrate that it unifies that of relative Rota-Baxter operators (of any weight), twisted Rota-Baxter operatots and Reynolds operators, and deformation map of matched pair of Lie triple systems. Parallel to the case of $\huaD$-maps, we should construct the 0-cocycle of $\frkD$-maps first.

\begin{pro}\label{reps}
Let $B:\g_2\longrightarrow\g_1$ be a $\frkD$-map of a quasi-twilled Lie triple system $(\huaG,\g_1,\g_2)$. Then
$$\Courant{u,v,w}_B:=\Courant{u,v,w}_2+\Courant{B(u),B(v),B(w)}_2+\Courant{B(u),B(v),w}_2+\Courant{u,B(v),B(w)}_2-\Courant{v,B(u),B(w)}_2$$
is a Lie triple system structure on $\g_2$ and linear map $\varrho:\ot^2\g_2\longrightarrow\gl(\g_1)$ defined to be
$$\varrho(u,v)x:=\Courant{x,u,v}_1+\Courant{x,B(u),B(v)}_1-B\Courant{x,B(u),B(v)}_2-B\Big(\Courant{x,B(u),v}_2-\Courant{u,x,B(v)}_2\Big)$$
is a representation of Lie triple system $(\g_2,\Courant{\cdot,\cdot,\cdot}_B)$ on the vector space $\g_1$.
\end{pro}
\begin{proof}
It follows from Proposition \ref{twilled} directly.
\end{proof}

\begin{pro}
Let $B:\g_2\longrightarrow\g_1$ be a $\frkD$-map of a quasi-twilled Lie triple system $(\huaG,\g_1,\g_2)$. Define a linear map $\delta:\wedge^2\g_1\longrightarrow\Hom(\g_2,\g_1)$ to be
$$\delta(x,y)v:=B\Courant{x,y,v}_2+B\Courant{x,y,B(v)}_2-\Courant{x,y,B(v)}_1,\quad \forall x,y\in \g_1,v\in \g_2.$$
Then $\delta(x,y)$ is a $1$-cocycle of the Lie triple system $(\g_2,\Courant{\cdot,\cdot,\cdot}_B)$ associated to the representation $(\g_1;\varrho)$.
\end{pro}
\begin{proof}
Similar to proof of Proposition \ref{1-cocy}, it is only needed to show that for all $x,y\in \g_1,u,v,w\in \g_2$,
\begin{eqnarray*}
~ &&\de(\delta(x,y))(u,v,w)\\
~ &=&\varrho(v,w)\delta(x,y)u-\varrho(u,w)\delta(x,y)v+D_\varrho(u,v)\delta(x,y)w-\delta(x,y)\Courant{u,v,w}_B\\
~ &=&0.
\end{eqnarray*}
By using the fact that $B$ is a $\frkD$-map and expanding the expressions of $\varrho$ and $\Courant{\cdot,\cdot,\cdot}_B$ given in Proposition \ref{reps}, we obtain the conclusion directly.
\end{proof}

Now we are ready to define the cohomology of $\frkD$-map. Let $B:\g_2\longrightarrow\g_1$ be a $\frkD$-map of a quasi-twilled Lie triple system $(\huaG,\g_1,\g_2)$. Define the space of $n$-cochains $C^n(B)$ to be
\begin{eqnarray*}
C^n(B):=
\begin{cases}
\wedge^2\g_1, & n=0,\\
C^{n}_{\mathsf{Ymg}}(\g_2;\g_1) & n\geq 1.
\end{cases}
\end{eqnarray*}

The coboundary map $\delta_B:C^n(B)\longrightarrow C^{n+1}(B)$ is defined to be
\begin{eqnarray*}
\delta_B:=
\begin{cases}
\delta,\quad n=0,\\
\de, \quad n\geq 1.
\end{cases}
\end{eqnarray*}
If $n\geq 1$, the corresponding cochain complex $(\oplus_{n=1}^{+\infty}C^n_{\mathsf{Ymg}}(\g_2;\g_1),\de)$ is the Lie triple system $(\g_2,\Courant{\cdot,\cdot,\cdot}_B)$ associated with the representation $(\g_1;\varrho)$ constructed in Proposition \ref{reps}.

Let $(\huaG,\g_1,\g_2)$ be a quasi-twilled Lie triple system and $B:\g_2\longrightarrow\g_1$ a $\frkD$-map of $(\huaG,\g_1,\g_2)$. Then we obtain the cochain complex $(C^*(B):=\oplus_{n=0}^{+\infty}C^n(B),\delta_B)$, whose cohomology is defined to be the {\bf cohomology of the $\frkD$-map $B$}. Denote by $Z^n(B)$ the set of $n$-cocycles and denote by $B^n(B)$ the set of $n$-coboundaries. The $n$-th cohomology group is taken to be
$$H^n(B):=Z^n(B)/B^n(B),\quad n\geq 0.$$

\emptycomment{
In the sequel, we given an intrinsic interpretation of the above coboundary map.
\begin{thm}
Let $B:\g_2\longrightarrow\g_1$ be a $\frkD$-map of a quasi twilled Lie triple system $(\huaG,\g_1,\g_2)$. Then for any $f\in C^n(\g_2,\g_1)$, we have
$$l_1^B(f)=(-1)^{k-1}\delta_B(f),$$
where $l_1^B$ is the differential of the twisted $L_\infty$-algebra $(C^*(\g_2,\g_1),l_1^B,l_2^B,l_3^B,l_4^B)$.
\end{thm}
\begin{proof}
for any $f\in C^n(\g_2,\g_1)$, we have
\begin{eqnarray*}
l_1^B(f)&=&l_1(f)+\frac{1}{2}l_3(B,B,f)+\frac{1}{6}l_4(B,B,B,f)\\
~ &=&[\hat\mu_2,\hat f]_{\mathsf{LTS}}+\frac{1}{2}[[[\hat\mu_1,\hat B]_{\mathsf{LTS}},\hat B]_{\mathsf{LTS}},\hat f]_{\mathsf{LTS}}+\frac{1}{6}[[[[\hat\phi_1,\hat B]_{\mathsf{LTS}},\hat B]_{\mathsf{LTS}},\hat B]_{\mathsf{LTS}},\hat f]_{\mathsf{LTS}}\\
~ &=&[\hat\mu_2+\frac{1}{2}[[\hat\mu_1, \hat B]_{\mathsf{LTS}},\hat B]_{\mathsf{LTS}}+\frac{1}{6}[[[\hat\phi_1,\hat B]_{\mathsf{LTS}},\hat B]_{\mathsf{LTS}},\hat B]_{\mathsf{LTS}},\hat f]_{\mathsf{LTS}}\\
~ &=&[\hat\mu_2^B,\hat f]_{\mathsf{LTS}}\\
~ &=&(-1)^{k-1}\delta_B(f).
\end{eqnarray*}
This finishes the proof.
\end{proof}
}

Cohomology of deformation maps of type II recovers cohomologies of relative Rota-Baxter operators (of any weight), twisted Rota-Baxter operators and Reynolds operators on Lie triple systems, and deformation maps of matched pairs of Lie triple systems when the corresponding deformation maps of type II are restricted to Examples \ref{crosseddd}-\ref{deforr} respectively.

\begin{ex}
Consider the quasi-twilled Lie triple system $(\g\ltimes_\rho\h,\g,\h)$ given in Example \ref{crossedd}. Let $T:\h\longrightarrow\g$ be a relative Rota-Baxter operator of weight $\lambda$ of $\g$ with respect to the action $\rho$. Then $(\h,\Courant{\cdot,\cdot,\cdot}_T)$ is a Lie triple system, where the structure is given to be
$$\Courant{u,v,w}_T:=\lambda\Courant{u,v,w}_\h+D_\rho(Tu,Tv)w+\rho(Tv,Tw)u-\rho(Tu,Tw)v,\quad \forall u,v,w\in \h.$$
Moreover, the Lie triple system $(\h,\Courant{\cdot,\cdot,\cdot}_T)$ represents on the vector space $\g$ via
$$\varrho(u,v)x:=\Courant{x,Tu,Tv}_\g-T\Big(D_\rho(x,Tu)v-\rho(x,Tv)u\Big),\quad \forall x\in \g,~u,v\in \h$$

Define $\delta:\wedge^2\g\longrightarrow\Hom(\h,\g)$ to be
$$\delta(x,y)v=T\Courant{x,y,v}_T+T\Courant{x,y,Tv}_T-\Courant{x,y,Tv}_\g,\quad \forall x,y\in \g,v\in \h.$$
Then the cohomology of $T$ is $(\oplus_{n=0}^{+\infty}C^n(T),\delta_T)$, where
$$C^n(T)=
\begin{cases}
\wedge^2\g, & n=0,\\
C^n_{\mathsf{Ymg}}(\h;\g), & n\geq 1,
\end{cases}$$
and
$$\delta_T=
\begin{cases}
\delta, & n=0,\\
\de,& n\geq1.
\end{cases}$$
\end{ex}

\begin{ex}
Consider the quasi-twilled Lie triple system $(\g\ltimes_\rho V,\g,V)$ given in Example \ref{semidirectt}. Let $T:V\longrightarrow\g$ be a relative Rota-Baxter operator on $\g$ with respect to a representation $(V;\rho)$. Then $(\h,\Courant{\cdot,\cdot,\cdot}_T)$ is a Lie triple system, where
$$\Courant{u,v,w}_T:=D_\rho(Tu,Tv)w+\rho(Tv,Tw)u-\rho(Tu,Tw)v,\quad \forall u,v,w\in V.$$
Moreover, the Lie triple system $(\h,\Courant{\cdot,\cdot,\cdot}_T)$ represents on the vector space $\g$ via
$$\varrho(u,v)x:=\Courant{x,Tu,Tv}_\g-T\Big(D_\rho(x,Tu)v-\rho(x,Tv)u\Big),\quad \forall x\in \g,~u,v\in V.$$

Define $\delta:\wedge^2\g\longrightarrow\Hom(\h,\g)$ to be
$$\delta(x,y)v=T\Courant{x,y,v}_T+T\Courant{x,y,Tv}_T-\Courant{x,y,Tv}_\g,\quad \forall x,y\in \g,v\in \h.$$
Then the cohomology of $T$ is $(\oplus_{n=0}^{+\infty}C^n(T),\delta_T)$, where
$$C^n(T)=
\begin{cases}
\wedge^2\g, & n=0,\\
C^n_{\mathsf{Ymg}}(\h;\g), & n\geq 1,
\end{cases}$$
and
$$\delta_T=
\begin{cases}
\delta, & n=0,\\
\de,& n\geq1.
\end{cases}$$
\end{ex}

\begin{ex}
Consider the quasi-twilled Lie triple system $(\g\ltimes_{\rho,\omega}V,\g,V)$ given Example \ref{twisted}. Let $T:V\longrightarrow\g$ be a twisted Rota-Baxter operator. Then $(\h,\Courant{\cdot,\cdot,\cdot}_T)$ is a Lie triple system, where
$$\Courant{u,v,w}_T:=D_{\rho}(Tu,Tv)w+\rho(Tv,Tw)u-\rho(Tu,Tw)v+\omega(Tu,Tv),\quad \forall u,v,w\in \h.$$
Moreover, the Lie triple system $(\h,\Courant{\cdot,\cdot,\cdot}_T)$ represents on the vector space $\g$ via
$$\varrho(u,v)x:=\Courant{x,Tu,Tv}_\g-T\Big(D_\rho(x,Tu)v-\rho(x,Tv)u-\omega(x,Tu,Tv)\Big),\quad \forall x\in\g, ~u,v\in \h.$$

Define $\delta:\wedge^2\g\longrightarrow\Hom(V,\g)$ to be
$$\delta(x,y)v=T\Courant{x,y,v}_T+T\Courant{x,y,Tv}_T-\Courant{x,y,Tv}_\g,\quad \forall x,y\in \g,v\in \h.$$
Then the cohomology of $T$ is $(\oplus_{n=0}^{+\infty}C^n(T),\delta_T)$, where
$$C^n(T)=
\begin{cases}
\wedge^2\g, & n=0,\\
C^n_{\mathsf{Ymg}}(V;\g), & n\geq 1,
\end{cases}$$
and
$$\delta_T=
\begin{cases}
\delta, & n=0,\\
\de,& n\geq1.
\end{cases}$$
\end{ex}

\begin{ex}
Consider the quasi-twilled Lie triple system $(\g\ltimes_{\frkR,\Courant{\cdot,\cdot,\cdot}}\g,\g,\g)$ given in Example \ref{Rey}. Let $R:\g\longrightarrow\g$ be a Reynolds operator. Then $(\g,\Courant{\cdot,\cdot,\cdot}_R)$ is a Lie triple system, where
$$\Courant{x,y,z}_R:=\Courant{Rx,Ry,z}+\Courant{x,Ry,Rz}-\Courant{y,Rx,Rz}+\Courant{Rx,Ry,Rz},\quad \forall x,,y,z\in \g.$$
The Lie triple system $(\g,\Courant{\cdot,\cdot,\cdot}_R)$ represents on the vector space $\g$ via
$$\varrho(x,y)z:=\Courant{x,Ry,Rz}-R\Big(\Courant{x,Ry,Rz}-\Courant{x,By,z}+\Courant{y,x,Bz},\quad \forall x,y,z\in \g.$$

Define $\delta:\wedge^2\g\longrightarrow\Hom(\g,\g)$ to be
$$\delta(x,y)v=R\Courant{x,y,z}_R+R\Courant{x,y,Rz}_R-\Courant{x,y,Rz}_\g,\quad \forall x,y,z\in \g.$$
Then the cohomology of $R$ is $(\oplus_{n=0}^{+\infty}C^n(R),\delta_R)$, where
$$C^n(R)=
\begin{cases}
\wedge^2\g, & n=0,\\
C^n_{\mathsf{Ymg}}(\g;\g), & n\geq 1,
\end{cases}$$
and
$$\delta_R=
\begin{cases}
\delta, & n=0,\\
\de,& n\geq1.
\end{cases}$$
\end{ex}

\begin{ex}
Let $(\g_1,\g_2;\rho_1,\rho_2)$ be a matched pair of Lie triple systems. Consider the quasi-twilled Lie triple system $(\g_1\bowtie\g_2,\g_1,\g_2)$ given in Example \ref{defor}. Let $B:\g_2\longrightarrow\g_1$ be a deformation map of $(\g_1,\g_2;\rho_1,\rho_2)$. Then $(\g_2,\Courant{\cdot,\cdot,\cdot}_B)$ is a Lie triple system, where
$$\Courant{u,v,w}_B:=\Courant{u,v,w}_{\g_2}+D_1(B(u),B(v))w+\rho_1(B(v),B(w))u-\rho_1(B(u),B(w))v,\quad \forall u,v,w\in \g_2.$$
The Lie triple system $(\g_2,\Courant{\cdot,\cdot,\cdot}_B)$ represents on the vector space $\g_1$ via for any $x\in \g_1,u,v\in \g_2,$
$$\varrho(u,v)x:=\rho_2(u,v)x+\Courant{x,B(u),B(v)}_{\g_1}-B\Courant{x,B(u),B(v)}_{\g_2}-B\Big(D_1(x,B(u))v-\rho_1(x,B(v))u\Big).$$

Define $\delta\wedge^2\g_1\longrightarrow\Hom(\g_2,\g_1)$ to be for any $x,y\in \g_1,v\in \g_2$,
$$\delta(x,y)v:=B\Big(D_1(x,y)v\Big)+B\Courant{x,y,B(v)}_{\g_2}-\Courant{x,y,B(v)}_{\g_1}.$$
Then the cohomology of deformation map $B$ is $(\oplus_{n=0}^{+\infty}C^n(B),\delta_B)$, where
\begin{eqnarray*}
C^n(B):=
\begin{cases}
\wedge^2\g_1, & n=0,\\
C^{n}_{\mathsf{Ymg}}(\g_2;\g_1) & n\geq 1,
\end{cases}
\end{eqnarray*}
and
\begin{eqnarray*}
\delta_B:=
\begin{cases}
\delta,\quad n=0,\\
\de, \quad n\geq 1.
\end{cases}
\end{eqnarray*}
\end{ex}

\begin{rmk}
As we all can see in the present paper, two types of deformation maps of quasi-twilled Lie triple systems recover several kinds of operators on Lie triple systems. We construct the 0-cochains and establish the cohomology of two types of deformation maps of quasi-twilled Lie triple systems, one thus can consequently use the cohomologies to classify the infinitesimal deformations of corresponding operators which we did not examined before. Moreover, we also consider to construct $L_\infty$-algebras whose Maurer-Cartan elements are precisely these two kinds of deformation maps. Consequently, by using the Maurer-Cartan elements, we obtain $L_\infty$-algebras that govern deformations of two types of deformation maps. We will consider this projection in the future and we also expect new studies in this direction.
\end{rmk}

{\bf Acknowledgment:}

This research is supported by NSFC grant 12401033 and China Scholarship Council.

\end{document}